\documentclass[a4paper,12pt]{article}
\usepackage[margin=1in]{geometry}
\usepackage{hyperref}
\usepackage{amsmath,amsthm}
\usepackage{graphicx}
\newtheorem{observation}{Observation}
\newtheorem{proposition}{Proposition}
\newtheorem{theorem}{Theorem}
\newtheorem{lemma}{Lemma}

\title{A short proof of an upper bound on the growth constant of polyiamonds}
\author{Vuong Bui\footnote{Swinburne Vietnam, FPT University, Hanoi, 80 Duy Tan Street, Hanoi 100000, Vietnam  (\href{mailto://bui.vuong@yandex.ru}{\texttt{bui.vuong@yandex.ru}})}}
\date{}

\begin{document}

\maketitle

\begin{abstract}
	We provide a short and elementary proof that the growth constant of polyiamonds is at most $1+2z+3z^2$ for the unique real root $z$ of the equation $2z^3+z^2-1=0$.
    This coincidentally suffices to recover the best known upper bound $3.6108$.
    Unlike the previous proof of this bound, which relied on computer-assisted technical arguments and the counts of polyiamonds with up to 75 triangles, our method is based on a straightforward recurrence that can be verified by hand with minimal effort.
\end{abstract}
\section{Introduction}
A polyiamond is an edge-connected set of triangles on the triangular lattice. There are two types of triangles, either to the left or to the right. See Fig.~\ref{fig:intro} for illustrations of the lattice with the two triangles, and the enumeration for the polyiamonds with up to 3 triangles. Two polyiamonds are considered to be the same and counted once if one is a translate of the other.
\begin{figure}[h]
\centering
\includegraphics[width=0.6\textwidth]{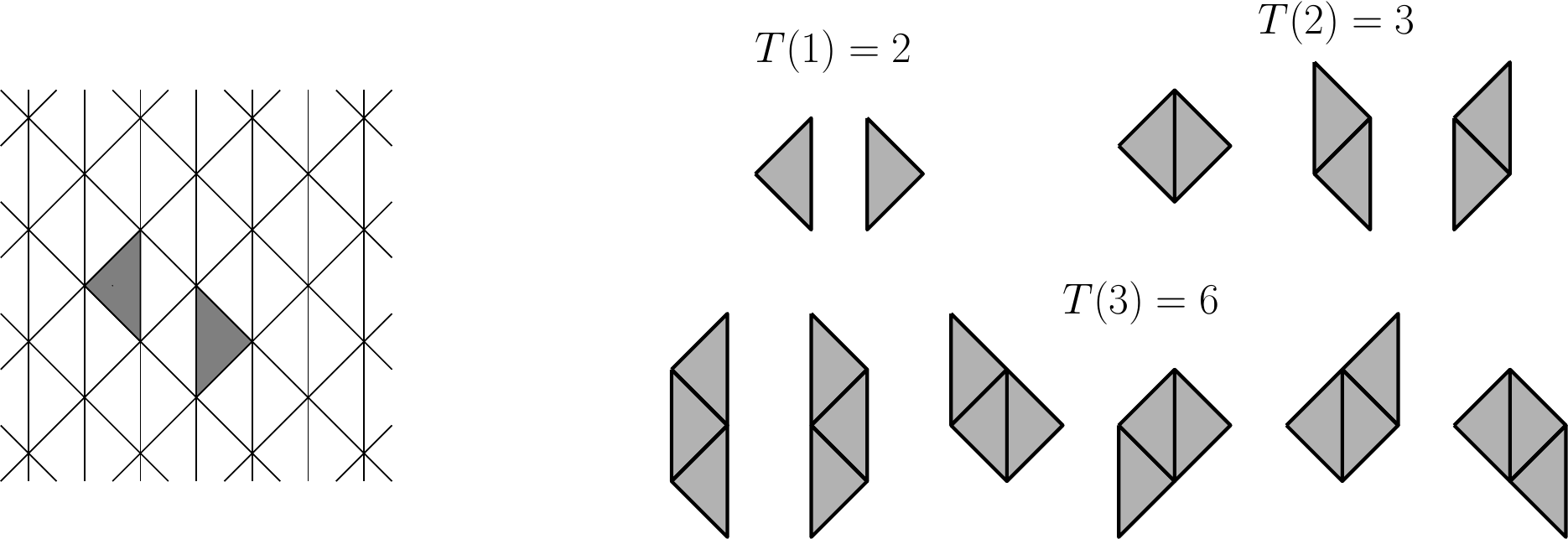}
\caption{Triangular lattice and polyiamonds with up to 3 triangles.}
\label{fig:intro}
\end{figure}

Let $T(n)$ denote the number of polyiamonds with $n$ triangles. The growth constant $\lambda_T=\lim_{n\to\infty}\frac{T(n+1)}{T(n)}$ was shown to exist in~\cite{madras1999pattern}. The supermultiplicativity of $T(n)$, in the sense that $T(\ell+m)\ge T(\ell)T(m)$ for every $\ell,m$ not both equal to $1$, was proved in~\cite{bui2023number}. It follows from Fekete's lemma that $\lambda_T=\lim_{n\to\infty}\sqrt[n]{T(n)}=\sup_n \sqrt[n]{T(n)}$. Plugging the largest known value $T(75)$ gives the best known lower bound 
\[
	\lambda_T\ge 2.8578.
\]
On the other hand, it took quite a while to reduce the trivial upper bound $4$ to the best known upper bound
\[
	\lambda_T\le 3.6108.
\]
The proof in~\cite{barequet2019improved} involves constraining the sizes of polyiamonds when we decompose a polyiamond into two smaller ones and studying the corresponding recurrences to upper bound $T(n)$. Like the situation of the lower bound, the proof uses the first 75 known values of $T(n)$ and only uses the recurrence to bound unknown values of $T(n)$ for $n>75$. In particular, they estimate the growth rate of the following upper bound $U(n)\ge T(n)$:
\[
	U(n) =
	\begin{cases}
		T(n) & \text{for $n \le 75$}, \\
		\left\lfloor \sum_{k=\lceil \frac{n-1}{3} \rceil}^{\lfloor\frac{2n+1}{3}\rfloor} \frac{(k+2)(n-k+2)}{4} U(k)U(n-k) + \frac{(n/2+2)^2}{4} U(\frac{n}{2}) \right\rfloor & \text{for $n > 75$}.
	\end{cases}
\]

Although the technique is quite heavy and the proof is computer-assisted, it is unfair to say that it is not efficient as the nature of such problems is usually hard. Let us look at the sibling problem for polyominoes, which are sets of edge-connected cells on the square lattice. While the corresponding growth rate $\lambda$, also known as Klarner's constant, was shown to be at most $4.649551$ by studying millions of the so-called twigs by Klarner and Rivest~\cite{klarner1973procedure} in 1973, it took almost half a century for the next improvement with $\lambda\le 4.5252$ by Barequet and Shalah~\cite{barequet2022improved} in 2022, using trillions (i.e. millions of millions) of twigs with
several additional tricks. Meanwhile, empirical estimates suggest that $\lambda$ is slightly larger than $4$, see \cite{sykes1976percolation}. There is an elementary approach without computer assistant in \cite{bui2025bounding} but it is not quite successful yet as it only proves $\lambda\le 4.83$. In other words, both polyominoes and polyiamonds seem to suffer the same situation. Note that $\lambda_T$ is believed to be slightly larger than $3$, see \cite{sykes1976percolation}. It is interesting that the estimated growth rates are slightly larger than the degrees of vertices (also known as the coordination numbers of the lattices) in these cases. However, this is not always the case since the growth rate of polyhexes is estimated to be quite less than $6$ at around $5.19$, see also \cite{sykes1976percolation}.

While simple recurrences do not yet provide an efficient upper bound on $\lambda$, the following theorem, whose proof uses elementary recurrences, yields the best known upper bound on $\lambda_T$.
\begin{theorem}\label{thm:main}
	The growth rate $\lambda_T$ of polyiamonds is at most $1+2z+3z^2$ for the unique real root $z$ of the equation $2z^3+z^2-1=0$.
\end{theorem}
One can see that the polynomials are quite mnemonic. In particular, $2z^3+z^2-1=0$ can be written as $\sum_{i=0}^3 (i-1)z^i = 0$, while $1+2z+3z^2 = \sum_{i=0}^2 (i+1)z^i$. They suggest that the recurrences should be simple.

The unique real root of the equation is 
\[
	z = \frac{1}{6} \left(-1 + \sqrt[3]{53 + 6\sqrt{78}} + \sqrt[3]{53 - 6\sqrt{78}} \right),
\]
which leads to an interesting coincidence\footnote{It is merely coincidence as the two approaches are not related.} that
\[
	\lambda_T\le 1+2z+3z^2 < 3.6108.
\]
The latter (strict) inequality is actually almost an equality. Although our bound is closer to $3.6107$ than $3.6108$, it would not be quite fair to say that we have improved the bound, since the work \cite{barequet2019improved} could already round the result up.
Instead, the main merit of this article is a very simple recurrence to prove the bound, which is elementary, short, and can be easily checked by hand, as in Section~\ref{sec:proof}.

For simplicity, we use the dual representation of polyiamonds, as a set of vertices connected by edges on the hexagonal lattice, see Fig.~\ref{fig:hexlat}. For example, one can see that $T(3)$ is $6$, because there are $6$ ways of connecting $3$ vertices using $2$ edges, nicely corresponding to $6$ angles of a hexagon. One can also check out the $6$ polyiamonds in the triangular lattice in Fig.~\ref{fig:intro}. Note that $T(n)$ are the same in both representations, except for $T(1)=2$ in the triangular lattice and $T(1)=1$ in the hexagonal lattice. (We cannot distinguish the left and right triangles using only $1$ vertex.) However, it has no effect on the asymptotic behavior. The dual representations have already been discussed at the end of~\cite{bui2023number}.

\begin{figure}[h]
\centering
\includegraphics[width=0.5\textwidth]{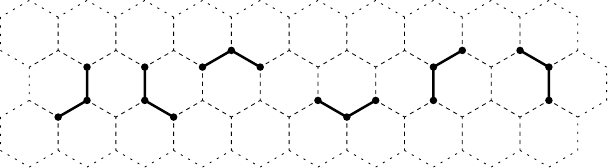}
\caption{Hexagonal lattice and the representation of $6$ polyiamonds of size $3$.}
\label{fig:hexlat}
\end{figure}

\section{Proof of Theorem~\ref{thm:main}}
\label{sec:proof}
Let $G_n$ denote the number of polyiamond-vertex pairs $(P,c)$ so that $P$ is a polyiamond with $n$ vertices and $c$ is a specially marked vertex that belongs to $P$ at the black bullet as in Fig.~\ref{fig:ghk}-(g) (which has a vertex vertically right below it). All the crossed vertices are not allowed to be included in $P$. We call these crossed vertices ``forbidden vertices''. We do not constrain any other vertices, and they may or may not belong to $P$. (The white bullets are for clearer arguments later, we do not constrain anything on them. They are merely non-forbidden neighbors of $c$.)
Let $H_n,K_n$ be the quantities corresponding to the Fig.~\ref{fig:ghk}-(h),(k), respectively.
Note that we allow $n$ to be zero. In this situation, the vertex at the black bullet does not belong to the (empty) polyiamond. We assume
\[
	G_0=H_0=K_0=1.
\]

\begin{figure}[h]
\centering
\includegraphics[width=0.75\textwidth]{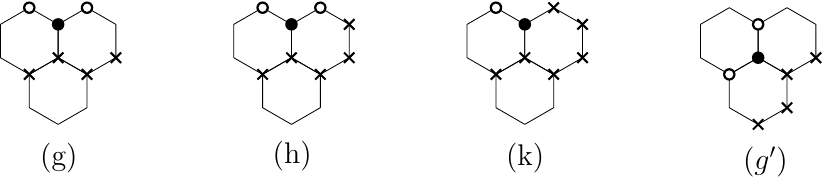}
\caption{Marked vertices and the surroundings.}
\label{fig:ghk}
\end{figure}

\begin{proposition}\label{prop:recur}
	For every $n\ge 1$,
	\[
		G_n \le \sum_{i+j=n-1} G_i H_j,\qquad H_n \le \sum_{i+j=n-1} G_i K_j,\qquad K_n \le G_{n-1},
	\]
        where $i,j$ range over nonnegative integers.
\end{proposition}
\begin{proof}
	In the case of $G_n$, the vertex $c$ at the black bullet is connected to the rest of the polyiamond $P$ via the two vertices at the white bullets. After excluding $c$ from the polyiamond, the remaining $P\setminus c$ either stays connected or gets decomposed into two polyiamonds. In either case, we can decompose the remaining into two polyiamonds of $i,j$ vertices so that each of them, if not empty, contains a vertex at the left, right white bullet, respectively. Obviously, $i+j=n-1$. Note that we allow here one polyiamond or both polyiamonds to be empty. An example for the former situation is when the vertex at the left white bullet does not belong to $P$ and $P\setminus c$ is connected and contains a vertex at the right white bullet. The latter situation, that is both polyiamonds are empty, happens when $n=1$. In total, we only constrain $i,j\ge 0$.

	Let us look at one of the two polyiamonds with the marked vertex at the left white bullet. The surrounding is now actually of Type (g') as in Fig.~\ref{fig:ghk}-(g'). (We can observe that the distance from the marked vertex to the closest forbidden vertex is 3.) We can rotate/flip the lattice to see that Types (g') and (g) are actually the same. Note that we even forbid more vertices than what are asked by Type (g'), but it does not matter for the sake of upper bounds. Arguing similarly, the surrounding for the right white bullet is of Type (h). Therefore,
	\[
		G_n \le \sum_{i+j=n-1} G_i H_j.
	\]

	The same kind of arguments works for the case of $H_n$ and $K_n$. We do not repeat it, except discussing a bit on the case of $K_n$. The black bullet here has only one non-forbidden neighbor, hence there is no convolution as in the previous cases with two non-forbidden neighbors. The new marked vertex (the white bullet) is of Type (g) as we can observe that the distance from the new marked vertex to the nearest forbidden vertex is 3.
\end{proof}

Now it becomes routine to upper bound the growth rate of $G_n$ as below.
\begin{lemma}
	The growth rate of $G_n$ is at most $1+2z+3z^2$ for the unique real root $z$ of the equation $2z^3+z^2-1=0$.
\end{lemma}
\begin{proof}
	By replacing the inequalities in Proposition~\ref{prop:recur} by the corresponding equalities, we define the following sequences $\hat{G},\hat{H},\hat{K}$ so that: $\hat{G}_0=\hat{H}_0=\hat{K}_0=1$ and for $n\ge 1$,
	\[
		\hat{G}_n = \sum_{i+j=n-1} \hat{G}_i \hat{H}_j,\qquad \hat{H}_n = \sum_{i+j=n-1} \hat{G}_i \hat{K}_j,\qquad \hat{K}_n = \hat{G}_{n-1}.
	\]
	It is obvious that for every $n$,
	\[
		G_n\le \hat{G}_n,\qquad H_n\le \hat{H}_n,\qquad K_n\le\hat{K}_n,
	\]
    and the calculation of the growth rate of $\hat{G}_n$ is standard as below.

	Consider the generating functions
	\[
		g(x)=\sum_{i\ge 0} \hat{G}_i x^i,\qquad h(x)=\sum_{i\ge 0} \hat{H}_i x^i,\qquad k(x)=\sum_{i\ge 0} \hat{K}_i x^i,
	\]
	which satisfy
	\[
		g(x) = 1+xg(x)h(x),\qquad h(x) = 1+xg(x)k(x),\qquad k(x)=1+xg(x).
	\]
	We eliminate $h(x),k(x)$ from the equations by writing
	\begin{multline*}
		g(x) = 1+xg(x)h(x) = 1+xg(x)(1+xg(x)k(x)) = 1+xg(x) + x^2g^2(x)k(x) \\
		= 1+xg(x) + x^2g^2(x)(1+xg(x)) = 1+xg(x) + x^2g^2(x) + x^3g^3(x).
	\end{multline*}
	Denoting $z(x) = xg(x)$, we have
	\[
		\frac{z(x)}{x} = 1 + z(x) + z^2(x) + z^3(x).
	\]
	Note that the coefficients of both generating functions $g(x)$ and $z(x)$ grow at the same rate and the radius of convergence $x$ must satisfy
	\begin{align*}
		\frac{z}{x} &= 1+ z + z^2 + z^3,\\
		\frac{1}{x} &= 1 + 2z + 3z^2,
	\end{align*}
	where $x,z>0$. (The second equation is the derivative by $z$ of the first equation.)
	Substituting $1/x$ in the first equation by the one in the second, we obtain
	\[
		z(1+2z+3z^2) = 1+z+z^2+z^3,
	\]
	which can be simplified to
	\[
		2z^3+z^2-1=0.
	\]
	It follows that the radius of convergence is $x=1/(1 + 2z + 3z^2)$ for the unique real root $z$.
\end{proof}

To finish the proof of Theorem~\ref{thm:main}, it suffices to observe the following.
\begin{observation}
	The growth rate of $G_n$ is an upper bound on the growth rate of the number of polyiamonds.
\end{observation}
\begin{proof}
	For a polyiamond $P$, among the bottom-most vertices of $P$ take the right-most vertex $c$. If $c$ has a vertex vertically right \emph{below} it, then $(P,c)$ is of Type (g). If $c$ has a vertex $d$ vertically right \emph{above} it instead, then $(P,c)$ is of Type (g') as in Fig.~\ref{fig:ghk}-(g'). Recall that Type (g) and Type (g') are equivalent, as one can be obtained from the other by rotating/flipping the lattice appropriately. Therefore, $T(n)\le 2G_n$. It follows that the growth rate of $G_n$ is an upper bound on the growth rate of polyiamonds.
\end{proof}
\subsection*{Acknowledgments}
The author is grateful to Gill Barequet for his useful comments on the manuscript.
\bibliographystyle{unsrt}
\bibliography{polyia36} 

\end{document}